\documentclass[11pt,reqno]{amsart}
\usepackage[margin=1.2in, footskip=.25in]{geometry}
\usepackage[T1]{fontenc}
\usepackage[american]{babel}
\usepackage[europeanresistors, cuteinductors]{circuitikz} 
\usetikzlibrary{decorations.markings}
\tikzstyle arrowstyle=[scale=1]
\tikzstyle directed=[postaction={decorate,decoration={markings,
    mark=at position .65 with {\arrow[arrowstyle]{stealth}}}}]
\tikzstyle reverse directed=[postaction={decorate,decoration={markings,
    mark=at position .65 with {\arrowreversed[arrowstyle]{stealth};}}}]
\usetikzlibrary{shapes.geometric}
\usepackage{amsmath,amsfonts,wasysym,bm}
\usepackage{mathrsfs}
\usepackage{graphicx,subcaption}
\usepackage{hyperref,cite}
\usepackage{caption}
\usepackage{float}
\usepackage{enumitem}
\usepackage{courier}
\usepackage[font=small]{caption}
\definecolor{mygreen}{RGB}{25,127,25}

\newtheorem{theorem}{Theorem}[section]
\newtheorem{proposition}[theorem]{Proposition}
\newtheorem{corollary}[theorem]{Corollary}
\newtheorem{lemma}[theorem]{Lemma}
\theoremstyle{definition}
\newtheorem{definition}{Definition}[section]

\newtheorem{remark}{Remark}[section]
\newtheorem{assumption}{Assumption}
\newtheorem*{thm*}{Theorem}
\newtheorem*{lem*}{Lemma}
\newtheorem*{prop*}{Proposition}

\newcommand{\mbbR}{{\mathbb R}}
\newcommand{\mbbN}{{\mathbb N}}
\newcommand{\mbbZ}{{\mathbb Z}}
\newcommand{\mbbC}{{\mathbb C}}

\newcommand{\mcB}{\mathcal{B}}
\newcommand{\mcE}{\mathcal{E}}
\newcommand{\mcF}{\mathcal{F}}

\def\mbbS{\mathbb{S}}
\def\mcA{\mathcal{A}}

\def\mcJ{\mathcal{J}}
\def\mcN{\mathcal{N}}

\newcommand{\mcH}{\mathcal{H}}

\def\sym{\operatorname{sym}}
\def\asym{\operatorname{sym}^\bot}
\def\id{\operatorname{id}}
\def\Pr{\operatorname{P}\!}
\numberwithin{equation}{section}

\title{Explicit formulas for heat kernels on diamond fractals}
\author{Patricia Alonso Ruiz}
\address{Department of Mathematics, University of Connecticut, Storrs, CT 06269}
\email{patricia.alonso-ruiz@uconn.edu}
\thanks{This research was partly supported by the NSF grant DMS-1613025 and the Feodor-Lynen Fellowship program from the Alexander von Humboldt Foundation.}

\begin{document}
\begin{abstract}
This paper provides explicit pointwise formulas for the heat kernel on compact metric measure spaces that belong to a $(\mbbN\times\mbbN)$-parameter family of fractals which are regarded as projective limits of metric measure graphs and do not satisfy the volume doubling property. The formulas are applied to obtain uniform continuity estimates of the heat kernel and to derive an expression of the fundamental solution of the free Schr\"odinger equation. The results also open up the possibility to approach infinite dimensional spaces based on this model.
\end{abstract}
\maketitle

\section{Introduction}\label{intro}
Many questions that arise in the study of the heat diffusion are closely related to the behavior of its associated heat kernel. This function, in the standard model of the Euclidean space $\mbbR^d$, is given by the Gaussian kernel 
\begin{equation}\label{E:Gaussian}
p_t(x,y)=\frac{1}{(4\pi t)^{d/2}}e^{-\frac{|x-y|^2}{4t}},
\end{equation}
which is the fundamental solution of the heat equation $\frac{\partial}{\partial t}u(t,x)=-\Delta u(t,x)$. Here, $\Delta$ denotes the standard Laplace operator on $\mbbR^d$. Any solution of the latter equation can thus be written as
\begin{equation}\label{E:solHeat}
u(t,x)=\int_{\mbbR^d} p_t(x,y)u_0(y)\,dy
\end{equation}
for any appropriate initial condition $u_0$. In general, from a probabilistic viewpoint, $p_t(x,y)$ is the transition density function a diffusion process with respect to some, often canonical, measure on the underlying space. As such, a heat kernel need not exist, and in spaces modeling disoredered media like fractals, this question is subject of extensive work~\cite{Kus85,HK03,Gri06,Kig12}. Even in more classical settings, the existence of the heat kernel is obtained in a fairly abstract way, where sometimes the kernel can be expressed in integral form; see e.g.~\cite{LP03,AG16,DEM16}. Providing a formula like~\eqref{E:Gaussian} or any other useful explicit expression is often a very difficult, if not an impossible task.

Both heat kernels and structures with fractal properties are present in the study of many physical phenomena. Among others, the former appear in the canonical partition function from statistical mechanics. In connection to quantum mechanical models on fractals~\cite{Str09,FKS09}, heat kernels are also used to define suitable Coulomb potentials. We refer to~\cite{RS11,Dun12,ADL13} and references therein for further applications and hints of fractal-like features in quantum mechanics, quantum gravity and wave propagation.

The present paper aims to provide explicit formulas for heat kernels on a family of compact metric measure spaces that arise as the scaling limit of \textit{generalized diamond hierarchical lattices}. These complicated spaces may lack many standard features such as differentiability or the volume doubling property~\eqref{E:doubling} and yet reveal a favorable intrinsic structure that will allow us to overcome these difficulties. 

Diamond hierarchical lattices have since long been studied in the physics literature in relation to statistical mechanics models~\cite{BO79}, spin systems and random polymers~\cite{KG82,KG84}, as well as Ising and Potts models~\cite{BZ88,CD89}. The interest in these lattices lies in the fact that their structure is richer than trees and simpler but in some sense closer to $\mbbZ^d$, making them useful to get an intuition for understanding physical models in $\mbbZ^d$. For instance it is known that the Ising model is solvable on them~\cite{Yan88}. Recent investigations in the context of random polymers, Ising and Potts models can be found in~\cite{HOF15,AC16,GM17}.

A generalized diamond lattice is characterized by two sequences of parameters $\mcJ=\{j_i\}_{i\geq 0}$ and $\mcN=\{n_i\}_{i\geq 0}$ that describe its branching properties; see Definition~\ref{D:DF.PL01}. The construction follows a recursive pattern as indicated in Figure~\ref{F:Intro01} and it gives rise to a sequence of graphs that converges after proper scaling to a compact subset of $\mbbR^2$. We call this set the \textit{diamond fractal} with parameters $\mcJ$ and $\mcN$. 
In the case when $j_i=j$ and $n_i=n$ for some fixed $j,n\geq 2$ and all $i\geq 1$, spectral properties of these spaces have been investigated in~\cite{ADT09}. Their spectral dimension, which in the fractal setting determines the leading term in the small time expansion of the heat kernel, follows the same known formula for so-called post critically finite sets~\cite[Theorem 2.4]{KL93} despite of lacking that property. Moreover, diamond fractals of this particular class are isospectral to fractal strings~\cite{ST12,LvF13}.

In the special case $\!j_i\!=n_i\!=\!2$, the spectrum of the Laplacian on the corresponding diamond fractal was described in~\cite{B+08}. Further properties of the diffusion process, heat kernel estimates and related functional inequalities such as Poincar\'e and elliptic Harnack inequalities were investigated in~\cite{HK10} while studying diffusion on the scaling limit of the critical percolation cluster. Some relevant results obtained there are summarized in the last section of this paper.

In contrast to the approach taken in previous investigations, one of the novelties of this paper is to regard a diamond fractal as the \textit{inverse}/\textit{projective limit} $(F_\infty,\{\Phi_i\}_{i\geq 0})$ of a system of metric measure graphs $\{(F_i,\{\phi_k\}_{k\leq i})\}_{i\geq 0}$. This will be key to obtain explicit formulas for the heat kernel. Inverse limit spaces have been studied from an analytic-geometric point of view by Cheeger-Kleiner~\cite{CK13,CK15} and more recently, results concerning gradient estimates for the heat kernel in this setting have been achieved in~\cite{CJKS17}. However, diamond fractals do not fit in those frameworks because, due to their branching rate, they do not fulfill the \textit{volume doubling} assumption
\begin{equation}\label{E:doubling}
\mu_\infty(B_{2R}(x))\leq C\mu_\infty(B_R(x)).
\end{equation}
To tackle this problem, we follow a more direct approach: firstly, a careful analysis of the diffusion process on each $F_i$ via \textit{cable systems/quantum graphs}~\cite{Var85,BB04,Kuc04} will allow us to obtain an explicit formula for each associated heat kernel $p^{F_i}_t(x,y)$, c.f.\ Theorem~\ref{T:MR02}. Secondly, a general construction argument for Markov processes in the inverse limit setting~\cite{ES01,BE04,Ste13} will lead to an expression of the heat kernel on $F_\infty$, c.f.\ Theorem~\ref{thm:MR03}.

Besides being interesting on their own, we would like to point out three main implications of these formulas that may open up the possibility to approach some questions involving heat kernels in this and related structures. 

As stated in Corollary~\ref{C:MR01}, Theorem~\ref{T:MR02} yields for each fixed $t>0$ and $x,y\in F_\infty$ the uniform bound
\begin{equation*}
|p_t^{F_i}(\Phi_i(x),\Phi_i(y))-p_t^{F_{i-1}}(\Phi_{i-1}(x),\Phi_{i-1}(y))|\leq N_iJ_i\big(1+(J_i^2t)^{-1}\big)e^{-J_i^2t},
\end{equation*}
where $N_i,J_i\to\infty$ as $i\to\infty$. For suitable $N_i$ and $J_i$, this implies the joint continuity of the heat kernel on $F_\infty$, c.f.\ Remark~\ref{R:HSD01}. In addition, this type of bound may be useful to study the spectral zeta function of the Laplacian~\cite{DGV08,CTT17} and functional inequalities such as \textit{Bakry-\'Emery} gradient estimates in the context of limits of metric graphs~\cite{BK17}.

The second application appears in Corollary~\ref{C:MR02} and it refers to the connection between the heat and the Schr\"odinger operators through the so-called \textit{Wick's rotation method}. While this method offers little help with estimates, it can turn to be very useful when exact formulas are available: if $u(t,x)$ is a function as in~\eqref{E:solHeat}, performing the change of variables $t\mapsto {\rm i}t$, which is the essence of Wick's rotation, we obtain a function $\tilde{u}(t,x):=u({\rm i}t,x)$ that satisfies
\begin{equation}\label{E:Schroedinger}
\frac{\partial}{\partial t}\tilde{u}(t,x)=-{\rm i}\Delta \tilde{u}(t,x).
\end{equation}
Setting $\psi_t(x,y)\!:=p_{{\rm i} t}(x,y)$ leads to the fundamental solution of the \textit{free Schr\"odin\-ger equation}~\eqref{E:Schroedinger}. In this manner, Theorem~\ref{thm:MR03} provides access to investigate  further related questions of interest in mathematical physics in the context of fractals~\cite{DABK83,Jum09,Akk13,CMT15}. 

The last observation is meant to motivate future research with a longer horizon and it is based on the fact that  all results obtained in this paper hold for quite general sequences $\mcJ$ and $\mcN$. In the particular case when $j_i=j$ and $n_i=n$ for some fixed $j,n\geq 2$ and all $i\geq 1$, the Hausdorff dimension of the diamond fractal can readily be seen to equal
\begin{equation*}\label{E:d_H}
\dim_H F_\infty=\frac{\log nj}{\log j}=1+\frac{\log n}{\log j}.
\end{equation*}
Roughly speaking, allowing $j$ and $n$ to grow so that $\frac{\log n}{\log j}$ diverges, one could obtain an infinite dimensional space approximated by a sequence of diamond fractals with these parameters. This idea resembles the approximation techniques used to approach the classical path integral formulation of quantum mechanics~\cite{AH76,JL00,AGHH12}.

The paper is organized as follows: Section~\ref{sec:1} describes the basic metric-measure aspects of diamond fractals, defining them as projective limits of a specific inverse limit system and introducing a cell structure to treat them analytically. The main results of the paper are stated in Section~\ref{sec:2}, c.f.\ Theorem~\ref{T:MR02} and Theorem~\ref{thm:MR03}, where the explicit formulas for the heat kernels are provided. The remaining sections are devoted to the proof of these results. Section~\ref{sec:3} gives several characterizations of the function spaces involved and an orthogonal decomposition based on average projections that is key in the subsequent proofs.
 Section~\ref{sec:4} studies the semigroup determined by the kernel function from Theorem~\ref{T:MR02}, and 
Section~\ref{sec:5} does the corresponding analysis for Theorem~\ref{thm:MR03}. In this last section we also take a closer look at the standard $2$-$2$ diamond fractal; we show that the present construction coincides with previous approaches and review some properties of the Dirichlet form and the heat kernel that are known in the literature for this particular case.

\section{Diamond fractals as inverse limits}\label{sec:1}
In this section we set up notation and explain how diamond fractals can be regarded as inverse/projective limit spaces of metric measure graphs. Diamond fractals do not satisfy the volume doubling condition and are therefore not covered by the framework investigated in~\cite{CK13,CK15}.  We start by recalling the general definition of inverse/projective limit systems of measure spaces, and afterwards discuss the particular case of diamond fractals.

\subsection{Inverse limits of measure spaces}
For further details concerning the abstract theory of inverse limits we refer to~\cite[Section 2-14]{HY61}.
\begin{definition}\label{def:Pre01}
Let $\{(F_i,\mu_i)\}_{i\geq 0}$ be a family of measure spaces and let $\{\phi_{ik}\}_{k\leq i}$ be a family of mappings such that
\begin{itemize}[leftmargin=.25 in]
\item[(a)] $\phi_{ik}\colon F_i\to F_k$ is measurable,
\item[(b)] for any $k\leq j\leq i$, $\phi_{jk}{\circ}\phi_{ij}=\phi_{ik}$ and $\phi_{ii}=\id$,
\item[(c)] for any $k\leq i$, $\mu_k=\mu_i{\circ}\phi_{ik}^{-1}$.
\end{itemize}
The collection $\{(F_i,\mu_i,\{\phi_{ik}\}_{k\leq i})\}_{i\geq 0}$ is called an inverse limit (or projective) system of measure spaces relative to $\{\phi_{ki}\}_{k\leq i}$.
\end{definition}
A more general definition of a projective system of measure spaces can be found e.g. in~\cite[2-14]{HY61}.
\begin{definition}\label{def:Pre02}
The inverse (or projective) limit space of $\{(F_i,\mu_i,\{\phi_{ik}\}_{k\leq i})\}_{i\geq 1}$ is defined as $(F_\infty,\mu_\infty,\{\Phi_i\}_{i\geq 0})$, where 
\begin{itemize}[leftmargin=.25in]
\item[(a)] $F_\infty:=\lim_{\leftarrow}F_i:=\{\{x_i\}_{i\geq 1}~|~x_i\in F_i\text{ and }\phi_{ik}(x_i)=x_k~\forall~k<i\}$,
\item[(b)] $\Phi_i\colon F_\infty\to F_i$ is given by $\Phi_i(x):=x_i$,
\item[(c)] $\mu_\infty(\Phi_i^{-1}(A))=\mu_i(A)$ for any $\mu_i$-measurable $A\subseteq F_i$.
\end{itemize}
\end{definition}
The space $F_\infty$ is Hausdorff, locally compact and second countable.
\begin{remark}\label{rem:Pre01}
By definition, $\phi_{ik}{\circ}\Phi_i=\Phi_k$ for any $0\leq k<i<\infty$.
\end{remark}
\subsection{Diamond fractals}\label{S:DF}
Regular diamond fractals and lattices earned their name from of the original construction as aggregations of ``deformed'' square lattices (diamonds) with multiple branches. As metric measure graphs, one may as well consider the branches in the lattice to be arcs, turning the first diamond into a circle as shown in Figure~\ref{F:Intro01}. This new parametrization happens to be especially useful to derive and prove formulas because of the available knowledge about the heat kernel on a circle; see e.g.~\cite[Section 2.5]{BGL14}.

A family of diamond fractals will be characterized by the sequences $\mcJ=\{j_i\}_{i\geq 0}$ and $\mcN=\{n_i\}_{i\geq 0}$. At stage $i$, each branch from the previous stage has got $j_i$ bonds and between them there are $n_i$ new branches. 

\begin{figure}[H]
\centering
\begin{tabular}{ccc}
\begin{tikzpicture}
\node at ($(135:1.75)$) {$F_0$};
\node at ($(225:1.75)$) {\color{white}$B_1$};
\foreach \a in {0,180}{
\draw ($(\a:1)$) to[out=90+\a,in=0+\a] ($(90+\a:1)$) to[out=180+\a,in=90+\a] ($(180+\a:1)$);
\fill[color= blue] ($(\a:1)$) circle (1pt);
}
\end{tikzpicture}
&
\begin{tikzpicture}
\node at ($(135:1.75)$) {$F_1$};
\node at ($(225:1.75)$) {\color{white}$B_2$};
\foreach \a in {0,60,120,180,240,300}{
\draw ($(\a:1)$) to[out=90-55+\a,in=15+\a] ($(60+\a:1)$);
\draw ($(\a:1)$) to[out=90+\a,in=330+\a] ($(60+\a:1)$);
\draw ($(\a:1)$) to[out=90+65+\a,in=-90+5+\a] ($(60+\a:1)$);
\fill[color= blue] ($(\a+60:1)$) circle (1pt);
}
\end{tikzpicture}
&
\begin{tikzpicture}
\node at ($(135:1.75)$) {$F_2$};
\node at ($(225:1.75)$) {\color{white}$B_3$};
\foreach \a in {0,60,120,180,240,300}{
\draw ($(\a:1)$) to[out=30+\a,in=340+\a] ($(30+\a:1.15)$);
\draw ($(\a:1)$) to[out=50+\a,in=320+\a] ($(30+\a:1.15)$);
\draw ($(\a:1)$) to[out=70+\a,in=300+\a] ($(30+\a:1.15)$);
\draw ($(\a:1)$) to[out=\a-30,in=\a-340] ($(\a-30:1.15)$);
\draw ($(\a:1)$) to[out=\a-50,in=\a-320] ($(\a-30:1.15)$);
\draw ($(\a:1)$) to[out=\a-70,in=\a-300] ($(\a-30:1.15)$);
}
\foreach \a in {0,60,120,180,240,300}{
\draw ($(\a:1)$) to[out=70+\a,in=320+\a] ($(30+\a:1)$);
\draw ($(\a:1)$) to[out=90+\a,in=300+\a] ($(30+\a:1)$);
\draw ($(\a:1)$) to[out=110+\a,in=280+\a] ($(30+\a:1)$);
\draw ($(\a:1)$) to[out=\a-70,in=\a-320] ($(\a-30:1)$);
\draw ($(\a:1)$) to[out=\a-90,in=\a-300] ($(\a-30:1)$);
\draw ($(\a:1)$) to[out=\a-110,in=\a-280] ($(\a-30:1)$);
}
\foreach \a in {0,60,120,180,240,300}{
\draw ($(\a:1)$) to[out=110+\a,in=320+\a] ($(30+\a:.725)$);
\draw ($(\a:1)$) to[out=130+\a,in=300+\a] ($(30+\a:.725)$);
\draw ($(\a:1)$) to[out=150+\a,in=280+\a] ($(30+\a:.725)$);
\draw ($(\a:1)$) to[out=\a-110,in=\a-320] ($(\a-30:.725)$);
\draw ($(\a:1)$) to[out=\a-130,in=\a-300] ($(\a-30:.725)$);
\draw ($(\a:1)$) to[out=\a-150,in=\a-280] ($(\a-30:.725)$);
\fill[color= blue] ($(\a+60:1)$) circle (1pt);
\fill[color= blue] ($(30+\a:1)$) circle (1pt);
\fill[color= blue] ($(30+\a:1.15)$) circle (1pt);
\fill[color= blue] ($(30+\a:0.725)$) circle (1pt);
}
\end{tikzpicture}
\end{tabular}
\vspace*{-1em}
\caption{First 3 levels of the construction of the diamond fractal of parameters $j_1=3$, $n_1=3$, and $j_2=2$, $n_2=3$.}
\label{F:Intro01}
\end{figure}
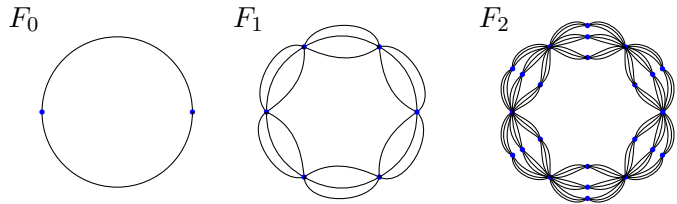
\subsubsection{Projective limit}
Let $\mbbS^1$ denote the one-dimensional sphere in the complex plane which we identify with the interval $[0,2\pi)$ via the mapping $\varphi_{\text{\o}}(\theta)=e^{\rm{i}\theta}$. For given parameter sequences $\mcJ$ and $\mcN$ we construct next the inverse limit system whose inverse limit space corresponds to a diamond fractal with these parameters. We start by introducing some notation. 
\begin{definition}\label{D:DF.PL01}
Set $j_0=n_0=1$ and $J_0=N_0=1$. For any sequences $\mcJ=\{j_\ell\}_{\ell\geq 0}$ and $\mcN=\{n_\ell\}_{\ell\geq 0}$, define
\begin{equation*}
J_i:=\prod_{\ell=1}^ij_\ell,\qquad\qquad N_i:=\prod_{\ell=1}^in_\ell\qquad\text{and}\qquad[n_i]=\{1,\ldots,n_i\}
\end{equation*}
for each $i\geq 1 $.
\end{definition}

We will assume that $j_\ell,n_\ell\geq 2$ for any $\ell\geq 1$. The spaces $F_i$ that build the system are constructed inductively, see Figure~\ref{F:DF.PL01} for a more visual description.

\begin{definition}\label{def:DF.PL01}
Let $\vartheta_0:=\{0,\pi\}$ and $\vartheta_i:=\big\{\frac{\pi k}{J_i}~|~0< k<2J_i,~k\hspace*{-.5em}\mod j_i\not\equiv 0\big\}$ for each $i\geq 1$. Moreover, let $B_0:=\vartheta_0$, $B_1=B_0\cup\vartheta_1$, and for each $i\geq 2$ define
\begin{equation*}\label{E:B_i}
B_i:=B_{i-1}\cup(\vartheta_i\times[n_1]\times\ldots\times[n_{i-1}]).
\end{equation*}
Further, let $F_0=\mbbS^1$ and for each $i\geq 1$ define $F_i:=F_{i-1}\times[n_i]/_\sim$, where $xw\stackrel{i}{\sim} x'w'$ if and only if $x,x'\in B_i$.
\end{definition}
The set $B_i$ contains the identification (branching, junction, vertex) points that yield $F_i$ and $B_i\subseteq F_{i-1}$. We will identify any point $x\in F_i$ with a word $\eta w_1\ldots w_\ell\in F_0\times[n_1]\times\ldots\times[n_\ell]$, $0<\ell\leq i$, in such a way that if $x\in F_i\setminus B_i$, then $x=\eta w_1\ldots w_i$ with $\eta\in F_0\setminus \cup_{k=0}^i\vartheta_k$, whereas if $x\in B_\ell\setminus B_{\ell-1}$, then $x=\eta w_1\ldots w_{\ell-1}$ and $\eta\in\vartheta_\ell$.

\begin{figure}[H]
\begin{tabular}{ccccc}
\begin{tikzpicture}
\node at ($(135:1.75)$) {$F_0$};
\node at ($(225:1.75)$) {\color{blue}$B_0$};
\foreach \a in {0,180}{
\draw ($(\a:1)$) to[out=90+\a,in=0+\a] ($(90+\a:1)$) to[out=180+\a,in=90+\a] ($(180+\a:1)$);
\fill[color= blue] ($(\a:1)$) circle (1.5pt);
}
\end{tikzpicture}
&
\hspace*{-2em}
\begin{tikzpicture}
\node at ($(135:1.75)$) {\color{white}$F_0\times \{1,2,3\}$};
\node at ($(225:1.75)$) {\color{white}$B_1$};
\node at ($(190:1)$) {$\longrightarrow$};
\end{tikzpicture}
&
\hspace*{-2em}
\begin{tikzpicture}
\node at ($(135:1.75)$) {$F_0\times [3]$};
\node at ($(225:1.75)$) {\color{white}$B_1$};
\draw ($(90:.5)+(180:.5)$) ellipse (3em and .5em);
\draw ($(180:.5)$) ellipse (3em and .5em);
\draw ($(270:.5)+(180:.5)$) ellipse (3em and .5em);
\end{tikzpicture}
&
\hspace*{-2em}
\begin{tikzpicture}
\node at ($(135:1.75)$) {\color{white}$F_0\times [3]$};
\node at ($(225:1.75)$) {\color{white}$B_1$};
\node at ($(180:1)$) {$\xrightarrow{\quad\stackrel{1}{\sim}\quad}$};
\end{tikzpicture}
&
\hspace*{-2em}
\begin{tikzpicture}
\node at ($(135:1.75)$) {$F_1$};
\node at ($(225:1.75)$) {\color{blue}$B_1$};
\foreach \a in {0,60,120,180,240,300}{
\draw ($(\a:1)$) to[out=90-55+\a,in=15+\a] ($(60+\a:1)$);
\draw ($(\a:1)$) to[out=90+\a,in=330+\a] ($(60+\a:1)$);
\draw ($(\a:1)$) to[out=90+65+\a,in=-90+5+\a] ($(60+\a:1)$);
\fill[color= blue] ($(\a+60:1)$) circle (1.5pt);
}
\end{tikzpicture}
\end{tabular}
\caption{Construction of $F_1$ from $F_0$ with $j_1=3$, $n_1=3$ and $B_1=\{\frac{\pi k}{3}~|~0\leq k<6\}$.}
\label{F:DF.PL01}
\end{figure}
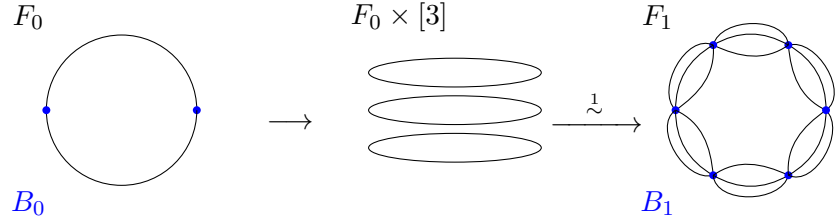
By means of the latter identification, a suitable family of mappings $\{\phi_{ik}\}_{k\leq i}$ can be defined as a ``shortening'' of words, where each $\eta w_1\ldots w_\ell\in F_i$ with $k<\ell\leq i$ becomes a word of length $k$, while words of length $\ell\leq k$ remain unchanged.
\begin{definition}\label{D:DF.PL02}
For any $0\leq k<i<\infty$, define the mapping $\phi_{ik}\colon F_i\to F_k$ as
\begin{equation*}\label{E:defphi_i}
\phi_{ik}(\eta w_1\ldots w_\ell)=
\begin{cases}
\eta w_1\ldots w_k&\text{if }k<\ell\leq i,\\
\eta w_1\ldots w_\ell&\text{if }0\leq\ell\leq k.
\end{cases}
\end{equation*}
\end{definition}

In particular, setting $\phi_i:=\phi_{i(i-1)}\colon F_i\to F_{i-1}$ we have that
\begin{equation*}\label{eq:DF.PL02}
\begin{cases}
\phi_i(x)=x&\text{if }x\in B_i,\\
\phi_i(xw)=x&\text{if }xw\in F_i\setminus B_i,
\end{cases}
\end{equation*}
and for each $0\leq k<i<\infty$, $\phi_{ik}=\phi_{k+1}{\circ}\cdots{\circ}\phi_i$. 

Together with a suitable family of measures $\{\mu_i\}_{i\geq 0}$ introduced in Definition~\ref{def:DF.CS03}, the family $\{(F_i,\mu_i,\{\phi_{ik}\}_{k\leq i})\}_{i\geq 0}$ defines a projective system of measure spaces with a limit as displayed in Figure~\ref{F:ProjLim}, where $\psi_i\colon F_{i-1}\times[n_i]\to F_{i-1}$ is the projection $\psi_i(xw)=x$, and the mapping $\pi_i\colon F_{i-1}\times[n_i]\to F_i$ is given by
\begin{equation*}\label{eq:DF.PL01}
\pi_i(xw)=\begin{cases}
xw&\text{if }x\in F_{i-1}\setminus B_i,\\
x&\text{if }x\in B_i.
\end{cases}
\end{equation*} 

\begin{definition}\label{def:DF03}
The inverse limit of the projective system $\{(F_i,\mu_i,\{\phi_{ik}\}_{k\leq i})\}_{i\geq 0}$ is called the diamond fractal with parameter sequences $\mcJ$ and $\mcN$.
\end{definition}

\begin{figure}[H]
\begin{tikzpicture}
\node (sp1) at ($(0:0)$) {$\vdots$};
\node (sp2) at ($(0:3.75)$) {$\vdots$};
\node (lim) at ($(0:6)$) {$F_\infty$};
\node (i) at ($(270:1)$) {$F_{i-1}\times[n_i]$};
\node (Fi) at ($(270:1)+(0:3.75)$) {$F_i$};
\node (im1) at ($(270:2.75)$) {$F_{i-2}\times[n_{i-1}]$};
\node (Fim1) at ($(270:2.75)+(0:3.75)$) {$F_{i-1}$};
\node (sp3) at ($(270:4)$) {$\vdots$};
\node (sp4) at ($(270:4)+(0:3.75)$) {$\vdots$};

\draw[->] (lim) -- node[above] {\small{$\Phi_i$}} (Fi);
\draw[->] (lim) -- node[right] {\small{$\Phi_{i-1}$}} (Fim1);
\draw[->] (i) -- node[above] {\small{$\pi_i$}} (Fi);
\draw[->] (im1) -- node[above] {\small{$\pi_{i-1}$}} (Fim1);
\draw[->] (i) -- node[above right] {\small{$\psi_i$}} (Fim1);
\draw[->] (Fi) -- node[left] {\small{$\phi_i$}} (Fim1) ;
\draw[->] (Fim1) -- node[left] {\small{$\phi_{i-1}$}} (sp4);
\end{tikzpicture}
\caption{Projective system for a diamond fractal.}
\label{F:ProjLim}
\end{figure}
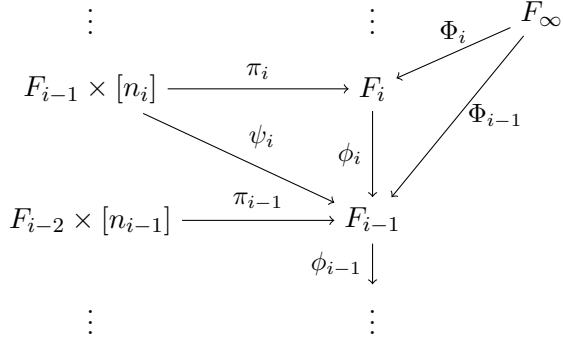
We denote the limit space by $(F_\infty,\mu_\infty,\{\Phi_i\}_{i\geq 0})$ and only refer to the parameters $\mcJ,\mcN$ if necessary.

\subsubsection{Cell structure}
In order to study functions and describe a diffusion process on $F_\infty$, it will be useful to see each approximating space $F_i$ as a union of $i$-cells isomorphic to intervals of length $\pi/J_i$, glued in an appropriate manner. These cells correspond to quantum graph edges and the cell structure fits into the more general framework of~\cite{Tep08}. Moreover, it induces a natural measure $\mu_i$ on each $F_i$, c.\ f.\ Definition~\ref{def:DF.CS03}, that makes $\{(F_i,\mu_i,\{\phi_{ik}\}_{k\leq i})\}_{i\geq 0}$ a projective system in accordance with Definition~\ref{def:Pre01}.

To set up this construction, recall the notation from Definition~\ref{def:DF.PL01}.
\begin{definition}\label{def:DF.CS01}
Let $\mcA_0:=\{\text{\o}\}$ and for each $i\geq 1$ define
\begin{equation*}\label{E:A_i}
\mcA_i:=\cup_{\ell=0}^i\vartheta_\ell\times[n_1]\times\ldots\times[n_i].
\end{equation*}
For any $i\geq 1$, let $L_i:=\pi /J_i$, and for each $\alpha\in\mcA_i$ define $\varphi_\alpha\colon [0,L_i]\to I_\alpha\subseteq F_i$ to be an isomorphism such that $\varphi_\alpha(0)$ and $\varphi_\alpha(L_i)$ are the endpoints of $I_\alpha:=\varphi_\alpha([0,L_i])$. We call $I_\alpha$ an $i$-cell of $F_i$. 
\end{definition}
For the sake of completeness, in the case $i=0$ we put $\varphi_{\text{\o}}(\theta)=e^{\rm{i}\theta}$, $L_0=2\pi$ and $I_{\text{\o}}=F_0=\mbbS^1$.
\begin{remark}\label{R:DF.CS01}
\begin{itemize}[leftmargin=1.5em, wide=0em]
\item[(i)] Notice that, for each $i\geq 1$, the set $\cup_{\ell=0}^i\vartheta_\ell$ has $2J_i$ elements and hence $\mcA_i$ consists of $2J_iN_i$ words of length $i+1$.
\item[(ii)] The endpoints $\varphi_\alpha(0)$ and $\varphi_\alpha(L_i)$ of a cell $I_\alpha$ correspond to identification points of $F_i$ and thus belong to $B_i$. We will establish the convention that a cell ``starts'' at $\varphi_\alpha(0)$ and travels counter-clockwise as indicated in Figure~\ref{F:DF.CS01}. In this way, each junction point $x\in B_i$ has $n_i$ cells starting and $n_i$ cells ending at it.
\end{itemize}
\end{remark}

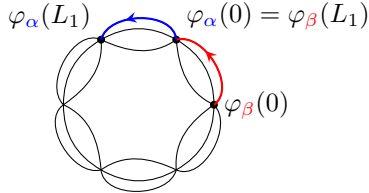
\begin{figure}[H]
\centering
\begin{tikzpicture}
\coordinate[label=above left:{\small{$\varphi_{{\color{blue}\alpha}}(L_1)$}}] (x) at ($(120:1)$);
\fill (x) circle (1.5pt);
\coordinate[label=above right:{\small{$\varphi_{{\color{blue}\alpha}}(0)=\varphi_{\color{red}\beta}(L_1)$}}] (z) at ($(60:1)$) ;
\fill (z) circle (1.5pt);
\coordinate[label=right:{\small{$\varphi_{\color{red}\beta}(0)$}}] (y) at ($(0:1)$);
\fill (y) circle (1.5pt);
\foreach \a in {0,60,120,180,240,300}{
\draw ($(\a:1)$) to[out=90-55+\a,in=15+\a] ($(60+\a:1)$);
\draw ($(\a:1)$) to[out=90+\a,in=330+\a] ($(60+\a:1)$);
\draw ($(\a:1)$) to[out=90+65+\a,in=-90+5+\a] ($(60+\a:1)$);
}
\draw[red, thick,directed] ($(0:1)$) to[out=90-55,in=15] ($(60:1)$); 
\draw[blue, thick, directed] ($(60:1)$) to[out=90-55+60,in=15+60] ($(120:1)$);
\end{tikzpicture}
\caption{In blue the cell $I_\alpha$, in red $I_{\beta}$, $n_1=3$.}
\label{F:DF.CS01}
\end{figure}

In the next definition, we associate each point in $F_i$ with the $i$-cell it belongs to and its position in it.
\begin{definition}\label{def:DF.CS02}
For each $i\geq 0$ and $x\in F_i$, define the pair $(\theta_x,\alpha_x)\in [0,L_i)\times\mcA_i$ to be such that
\begin{equation*}\label{eq:DF.CS01}
\varphi_{\alpha_x}(\theta_x)=x.
\end{equation*}
For each $x\in F_i$, the set that indexes the branches which belong to the same ``bundle'' as $I_{\alpha_x}$ is denoted by
\begin{equation*}\label{eq:DF.CS02}
\mcA_x:=\{\alpha\in \mcA_i~|~\varphi_\alpha(0)=\varphi_{\alpha_x}(0)\}.
\end{equation*}
\end{definition}

\begin{remark}\label{rem:DF.CS02}
The pair $(\theta_x,\alpha_x)$ is not unique if $x\in B_i$, $i\geq 1$. In that case we adopt the convention that $\alpha_x$ is the smallest in the lexicographic order.
\end{remark}
\begin{figure}[H]
\centering
\begin{subfigure}[b]{.4\textwidth}
\centering
\begin{tikzpicture}
\coordinate[label=right:{\small{$x$}}] (x) at ($(40:1.15)$);
\foreach \a in {0,60,120,180,240,300}{
\draw ($(\a:1)$) to[out=90-55+\a,in=15+\a] ($(60+\a:1)$);
\draw ($(\a:1)$) to[out=90+\a,in=330+\a] ($(60+\a:1)$);
\draw ($(\a:1)$) to[out=90+65+\a,in=-90+5+\a] ($(60+\a:1)$);
}
\draw[red,thick] ($(0:1)$) to[out=90-55,in=15] ($(60:1)$); 
\fill (x) circle (1.5pt);
\end{tikzpicture}
\caption{In red the cell $I_{\alpha_x}$.}
\label{F:DF.CS02a}
\end{subfigure}
\hspace*{.2in}
\begin{subfigure}[b]{.4\textwidth}
\centering
\begin{tikzpicture}
\coordinate[label=right:{\small{$x$}}] (x) at ($(40:1.15)$);
\foreach \a in {0,60,120,180,240,300}{
\draw ($(\a:1)$) to[out=90-55+\a,in=15+\a] ($(60+\a:1)$);
\draw ($(\a:1)$) to[out=90+\a,in=330+\a] ($(60+\a:1)$);
\draw ($(\a:1)$) to[out=90+65+\a,in=-90+5+\a] ($(60+\a:1)$);
}
\draw[red,thick] ($(0:1)$) to[out=90-55,in=15] ($(60:1)$); 
\draw[red,thick] ($(0:1)$) to[out=90,in=330] ($(60:1)$);
\draw[red,thick] ($(0:1)$) to[out=90+65,in=-90+5] ($(60:1)$);
\fill (x) circle (1.5pt);
\end{tikzpicture}
\caption{In red, cells $I_\alpha$ with $\alpha\in\mcA_x$.}
\label{F:DF.CS02b}
\end{subfigure}
\end{figure}

The cell structure naturally induces a measure on each $F_i$, that is obtained by redistributing the mass of each branch in the previous level uniformly between its ``successors''.
\begin{definition}\label{def:DF.CS03}
For each $i\geq 0$ and $\alpha\in\mcA_i$ define the measure $\mu_i$ on $I_\alpha$ as the push forward 
\begin{equation*}\label{eq:DF.CS03}
\int_{I_\alpha}f\,d\mu_i= \frac{1}{N_i}\int_0^{L_i}f{\circ}\varphi_\alpha(\theta)\,d\theta
\end{equation*}
for any measurable function $f\colon I_\alpha\to\mbbR$. On $F_i$, the measure $\mu_i$ is defined as
\begin{equation*}\label{eq:DF.CS04}
\int_{F_i}f\,d\mu_i=\sum_{\alpha\in\mcA_i}\int_{I_\alpha}f\,d\mu_i.
\end{equation*}
\end{definition}
From the latter definition, a direct computation and induction yield the relationship between measures from different levels.
\begin{lemma}\label{lem:DF.CS01}
For any $0\leq k<i<\infty$ and any Borel measurable function $f\colon F_k\to\mbbR$,
\begin{equation*}\label{eq:DF.CS06}
\int_{F_i}f{\circ}\phi_{ik}\,d\mu_i=\int_{F_k}f\,d\mu_k.
\end{equation*}
In particular, for any $i\geq 1$,
\begin{equation*}\label{eq:DF.CS05}
\int_{F_i}f{\circ}\phi_i\,d\mu_i=\int_{F_{i-1}}f\,d\mu_{i-1}
\end{equation*}
for any Borel measurable function $f\colon F_{i-1}\to\mbbR$.
\end{lemma}
The diamond fractal $F_\infty$ is thus equipped with the measure $\mu_\infty$ induced by the family $\{\mu_i\}_{i\geq 0}$ via Definition~\ref{def:Pre02}.
\section{Main results}\label{sec:2}
The natural diffusion process on a diamond fractal $F_\infty$ we want to analyze is based on a sequence of processes on the spaces $F_i$ constructed as described in~\cite{BE04}. In this section we present explicit formulas for the heat kernel of the Markov semigroup associated with both the approximating and the limiting processes.

Let $p_t^{F_0}\colon (0,\infty)\times F_0\times F_0\to\mbbR$ denote the heat kernel of the standard Brownian motion on $\mbbS^1$, i.e.
\begin{equation}\label{E:HKS1}
p_t^{F_0}(\theta_1,\theta_2)=\frac{1}{2\pi}+\frac{1}{\pi}\sum_{k=1}^\infty e^{-k^2 t}\cos(k(\theta_2-\theta_1)).
\end{equation}

Due to the distinguished role it will play in the proofs, we display separately the case $i=1$. Although the parameters $\mcJ=\{j_i\}_{i\geq 0}$ and $\mcN=\{n_i\}_{i\geq 0}$ of the diamond fractal $F_\infty$ are omitted in the notation of the kernel, they will appear inside the formulas.
\begin{theorem}\label{T:MR01}
Let $n_1,j_1\geq 2$. The heat kernel $p_t^{F_1}\colon (0,\infty)\times F_1\times F_1\to\mbbR$ on $F_1$ is given by
\begin{equation*}\label{eq:MR01}
p_t^{F_1}(x,y)=\begin{cases}
p_t^{F_0}(\phi_1(x),\phi_1(y))&\text{if }y=y_1,\\
p_t^{F_0}(\phi_1(x),\phi_1(y))-j_1\big(p^{F_0}_{j_1^2t}(j_1\theta_x,j_1\theta_y)-p^{F_0}_{j_1^2t}(j_1\theta_x,-j_1\theta_y)\big)&\text{if } y=y_2,\\
p_t^{F_0}(\phi_1(x),\phi_1(y))+(n_1-1)j_1\big(p^{F_0}_{j_1^2t}(j_1\theta_x,j_1\theta_y)-p^{F_0}_{j_1^2t}(j_1\theta_x,-j_1\theta_y)\big)&\text{if }y=y_3,
\end{cases}
\end{equation*}
where the possible pair-point configurations $(x,y)\in F_1$ are described in Figure~\ref{fig:MR01}.
\end{theorem}
The pair-point configurations correspond to the cases when $x$ and $y$ belong, respectively, to different cells in different ``bundles'' ($y=y_1$), to different cells in the same ``bundle'' ($y=y_2$) and to the same cell ($y=y_3$).
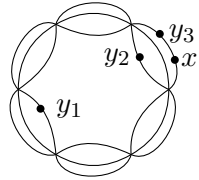
\begin{figure}[H]
\centering
\begin{tikzpicture}
\node (x) at ($(20:1.15)$) {$\quad x$};
\node (y1) at ($(200:0.75)$) {$\qquad y_1$};
\node (y2) at ($(35:0.75)$) {$y_2\;\;\quad$};
\node (y3) at ($(40:1.15)$) {$\quad\;\; y_3$};
\fill (x) circle (1.5pt);
\fill (y1) circle (1.5pt);
\fill (y2) circle (1.5pt);
\fill (y3) circle (1.5pt);
\foreach \a in {0,60,120,180,240,300}{
\draw ($(\a:1)$) to[out=90-55+\a,in=15+\a] ($(60+\a:1)$);
\draw ($(\a:1)$) to[out=90+\a,in=330+\a] ($(60+\a:1)$);
\draw ($(\a:1)$) to[out=90+65+\a,in=-90+5+\a] ($(60+\a:1)$);
}
\end{tikzpicture}
\caption{Pair-point configurations in the first approximation $F_1$ of a diamond fractal with parameters $j_1=3$ and $n_1=3$. }
\label{fig:MR01}
\end{figure}
In general, this classification of pair-point configurations can be spelled in terms of the corresponding word expression: for any $x,y\in F_i$ let us write $x=\eta w_1\ldots w_\ell$, $y=\eta'w'_1\ldots w'_{\ell'}$ for some $1\leq \ell\leq\ell'\leq i$ and $\eta,\eta'\in [0,2\pi)$.
\begin{enumerate}[leftmargin=*]
\item If $x$ and $y$ belong to different bundles ($y=y_1$), then we have the following cases:
\begin{enumerate}[wide=0in]
\item $\ell\neq\ell'$, i.e. $x,y\in B_i$ or $x\in B_i$ and $y\in F_i\setminus B_i$ or vice versa,
\item $\ell=\ell'<i$, i.e. $x,y\in B_i$,
\item $\ell=\ell'=i$, and $\eta\in[\vartheta,\vartheta+L_i)$, $\eta'\in[\vartheta',\vartheta'+L_i)$ for some $\vartheta,\vartheta'\in\cup_{k=0}^{i-1}\vartheta_k$ with $\vartheta\neq\vartheta'$, i.e. $x,y\in F_i\setminus B_{i-1}$,
\item $\ell=\ell'=i$, and $\eta,\eta'\in [\vartheta,\vartheta+L_i)$ for some $\vartheta\in\cup_{k=0}^{i-1}\vartheta_k$ and $w_{i-1}\neq w'_{i-1}$, i.e. $x,y\in F_i\setminus B_{i-1}$.
\end{enumerate}
\item If $x$ and $y$ belong to different strands of the same bundle ($y=y_2$), then $\ell=\ell'=i$, $\eta,\eta'\in [\vartheta,\vartheta+L_i)$ for some $\vartheta\in\cup_{k=0}^{i-1}\vartheta_k$, and $w_\ell= w'_\ell$ for all $1\leq \ell<i$ whereas $w_i\neq w'_i$.
\item If $x$ and $y$ belong to the same strand ($y=y_3$), then $\ell=\ell'=i$, $\eta,\eta'\in [\vartheta,\vartheta+L_i)$ for some $\vartheta\in\cup_{k=0}^{i-1}\vartheta_k$, and $w_\ell= w'_\ell$ for all $1\leq \ell\leq i$.
\end{enumerate}

Theorem~\ref{T:MR01} sets the basis of an inductive argument that will lead to the formula for the heat kernel on a generic approximating space $F_i$.
\begin{theorem}\label{T:MR02}
Let $\mcJ$ and $\mcN$ be parameter sequences. For each $i\geq 1$, the heat kernel on the corresponding $F_i$ is the function $p_t^{F_i}\colon (0,\infty)\times F_i\times F_i\to\mbbR$ given by
\begin{equation*}\label{eq:MR02}
p_t^{F_i}(x,y){=}\begin{cases}
p_t^{F_{i-1}}(\phi_i(x),\phi_i(y))&\text{if }y{=}y_1,\\
p_t^{F_{i-1}}(\phi_i(x),\phi_i(y))-N_{i-1}J_i\big(p^{F_0}_{J_i^{2}t}(J_i\theta_x,J_i\theta_y){-}p^{F_0}_{J_i^{2}t}(J_i\theta_x,-J_i\theta_y)\big)&\text{if }y{=}y_2,\\
p_t^{F_{i-1}}(\phi_i(x),\phi_i(y)){+}N_{i-1}(n_i{-}1)J_i\big(p^{F_0}_{J_i^{2}t}(J_i\theta_x,J_i\theta_y){-}p^{F_0}_{J_i^2t}(J_i\theta_x,{-}J_i\theta_y)\big)&\text{if }y{=}y_3,
\end{cases}
\end{equation*}
with the same pair-point configurations as in Theorem~\ref{T:MR01}.
\end{theorem}
Depending on the specific purpose, the iterative formula of Theorem~\ref{T:MR02} can be rewritten in different ways. Particularly insightful is the following: let $p_t^{[0,L_i]_D}$ denote the heat kernel on the interval $[0,L_i]$ subject to Dirichlet boundary conditions, i.e.
\begin{equation}\label{E:1dHK}
p_t^{[0,L_i]_D}(\theta_1,\theta_2)=\frac{2}{L_i}\sum_{k=1}^\infty e^{-\frac{k^2\pi^2}{L_i^2}t}\sin\Big(\frac{k\pi\theta_1}{L_i}\Big)\sin\Big(\frac{k\pi\theta_2}{L_i}\Big).
\end{equation}
Since $L_i=\pi/J_i$, the classical equality 
\begin{equation*}\label{E:HKHelp}
p_t^{[0,L_i]_D}(\theta_1,\theta_2)=J_ip^{[0,\pi]_D}_{J_i^2t}(J_i\theta_1,J_i\theta_2)=J_i\big(p^{F_0}_{J_i^{2}t}(J_i\theta_1,J_i\theta_2){-}p^{F_0}_{J_i^2t}(J_i\theta_1,{-}J_i\theta_2)\big)
\end{equation*}
yields the alternative formula of the heat kernel
\begin{equation}\label{E:MR02b}
p_t^{F_i}(x,y){=}\begin{cases}
p_t^{F_{i-1}}(\phi_i(x),\phi_i(y))&\text{if }y{=}y_1,\\
p_t^{F_{i-1}}(\phi_i(x),\phi_i(y))-N_{i-1}p_t^{[0,L_i]_D}(\theta_x,\theta_y)&\text{if }y{=}y_2,\\
p_t^{F_{i-1}}(\phi_i(x),\phi_i(y)){+}N_{i-1}(n_i{-}1)p_t^{[0,L_i]_D}(\theta_x,\theta_y)&\text{if }y{=}y_3.
\end{cases}
\end{equation}
Recall that by definition of projective limit, see Definition~\ref{def:Pre02}, any point $x\in F_\infty$ is approximated by a sequence of points $\{\Phi_i(x)\}_{i\geq 1}$, where $\Phi_i(x)\in F_i$. Consequently,~\eqref{E:MR02b} yields the following uniform estimate.
\begin{corollary}\label{C:MR01}
For any $i\geq 1$ and any $(t,x,y)\in (0,\infty)\times F_\infty\times F_\infty$ it holds that
\begin{equation}\label{E:unif estimate}
|p_t^{F_i}(\Phi_i(x),\Phi_i(y))-p_t^{F_{i-1}}(\Phi_{i-1}(x),\Phi_{i-1}(y))|\leq N_iJ_i\big(1+(J_i^2t)^{-1}\big)e^{-J_i^2t}.
\end{equation}
\end{corollary}
\begin{proof}
Applying~\eqref{E:MR02b} we have that
\begin{align*}
|p_t^{F_i}(\Phi_i(x),\Phi_i(y))&-p_t^{F_{i-1}}(\Phi_{i-1}(x),\Phi_{i-1}(y))|\leq  N_{i-1}(n_i-1)p_t^{[0,L_i]_D}(\theta_{\Phi_i(x)},\theta_{\Phi_i(y)})\\
&= N_{i-1}(n_i-1)\frac{2}{\pi}J_i\sum_{k=1}^\infty e^{-k^2J_i^2t}\sin\big(J_i k\theta_{\Phi_i(x)}\big)\sin\big(J_i k\theta_{\Phi_i(y)}\big)\\
&\leq \frac{2}{\pi}N_{i-1}(n_i-1)J_i\sum_{k=1}^\infty e^{-J_i^2kt}
\leq N_iJ_i\big(1+(J_i^2t)^{-1}\big)e^{-J_i^2t}.
\end{align*}
\end{proof}
This estimate is the reason why throughout the paper we will make the following standing assumption.
\begin{assumption}
The sequences $\mcJ$ and $\mcN$ satisfy
\begin{equation}\label{E:Ass_NJ}
\lim_{i\to\infty}N_ie^{-J^2_it}<\infty\qquad\qquad \forall~t\in(0,\varepsilon)
\end{equation}
for some $\varepsilon>0$.
\end{assumption}
We can now define for each fixed $t>0$ the uniform limit
\begin{equation*}\label{E:HKasLimit}
p_t^{F_\infty}(x,y):=\lim_{i\to\infty}p_t^{F_i}(\Phi_i(x),\Phi_i(y))\qquad\qquad x,y\in F_\infty
\end{equation*}
which is uniformly continuous on $F_\infty\times F_\infty$ and jointly continuous on any finite interval $0<t_1\leq t\leq t_2<\infty$, see Remark~\ref{R:HSD01}.
 The pointwise expression of the heat kernel on a diamond fractal $F_\infty$ is obtained by noticing that if $\Phi_i(x)$ and $\Phi_i(y)$ belong to $i$-cells in different bundles, then
\begin{equation*}
p_t^{F_i}(\Phi_i(x),\Phi_i(y))=p_t^{F_{i-1}}\big(\phi_i(\Phi_i(x)),\phi_i(\Phi_i(y))\big)=p_t^{F_{i-1}}(\Phi_{i-1}(x),\Phi_{i-1}(y)),
\end{equation*}
where last equality follows because $\phi_i(\Phi_i(x))=\Phi_{i-1}(x)$ for any $x\in F_\infty$, c.f. Definition~\ref{def:Pre02}.

\begin{remark}\label{R:MR01}
\begin{enumerate}[leftmargin=*,label=(\roman*), align=left, labelsep=0in,wide=0em]
\item The weaker assumption $\displaystyle\lim_{i\to\infty}N_ie^{-J^2_i\varepsilon}<\infty$ for some $\varepsilon>0$ yields the weaker result that $p^{F_\infty}_t(x,y)$ is uniformly continuous for $t\in[\varepsilon,\infty)$.
\item One can give stronger and still rather weak time-independent sufficient conditions that imply~\eqref{E:Ass_NJ}, for instance $\displaystyle \lim_{i\to\infty}N_ie^{-J_i}<\infty$.
\end{enumerate}
\end{remark}
\begin{theorem}\label{thm:MR03}
Let $\mcJ,\mcN$ be parameter sequences that satisfy~\eqref{E:Ass_NJ}. The heat kernel for the associated diamond fractal is the function $p_t^{F_\infty}\colon (0,\infty)\times F_\infty\times F_\infty\to\mbbR$ given by
\begin{equation}\label{eq:MR03}
p_t^{F_\infty}(x,y)=p_t^{F_{i_*}}\big(\Phi_{i_*}(x),\Phi_{i_*}(y)\big),
\end{equation}
where $i_*:=i_*(x,y):=\max_{i\geq 1}\{x\text{ and }y\text{ belong to i-cells in the same bundle}\}$.
\end{theorem}
In particular, one can write a formula for the fundamental solution of the free Schr\"odinger equation on $F_\infty$.
\begin{corollary}\label{C:MR02}
Let $\mcJ,\mcN$ be parameter sequences that satisfy~\eqref{E:Ass_NJ}. The free Schr\"odinger kernel for the associated diamond fractal is the function $\psi^{F_\infty}_t\colon(0,\infty)\times F_\infty\times F_\infty\to\mbbC$ given by
\begin{equation*}\label{E:CorSchroedinger}
\psi_t^{F_\infty}(x,y)=p_{{\rm i}t}^{F_{i_*}}\big(\Phi_{i_*}(x),\Phi_{i_*}(y)\big),
\end{equation*}
where $i_*:=i_*(x,y):=\max_{i\geq 1}\{x\text{ and }y\text{ belong to i-cells in the same bundle}\}$.
\end{corollary}

\begin{remark}\label{R:MR02}
Assumption~\eqref{E:Ass_NJ} is rather general and in particular includes unbounded sequences, whose geometric and analytic implications are subject of further research.
\end{remark}
\section{Function spaces}\label{sec:3}
In this section we introduce the function spaces we will be working with. Given any Hausdorff compact metric space $F$ we denote by $\mcB_b(F)$ the set of bounded Borel functions and by $C(F)$ the space of continuous real-valued functions, which in this case coincides with the space of continuous functions vanishing at infinity. Given another compact metric space $F'$ and a measurable function $\xi\colon F\to F'$, we define $\xi^*\colon \mcB_b(F')\to\mcB_b(F)$ by $\xi^*f:=f{\circ}\xi$. By definition of continuity, if $\xi\in C(F)$ and $\xi^{-1}(K)$ is compact for any compact $K\subseteq F'$, then $\xi^*\colon C(F')\to C(F)$.

\subsection{Finite approximations}
For each $i\geq 0$, let $L^2(F_i,\mu_i)$ denote the space of square integrable functions on $F_i$. For any $i\geq 1$, we decompose this space into
\begin{equation*}\label{eq:FS.FA01}
L^2(F_i,\mu_i)=L^2_{\sym}(F_i,\mu_i)\oplus L^2_{\asym}(F_i,\mu_i),
\end{equation*}
where $L^2_{\sym}(F_i,\mu_i)$ denotes the invariant subspace of $L^2(F_i,\mu_i)$ under the action of the symmetric group $S(n_i)^{2j_i}$. 

\begin{definition}\label{def:FS.FA01}
Let $i\geq 1$. Define the projection operator $\Pr_i\colon C(F_i)\to L^2_{\sym}(F_i,\mu_i)\cap C(F_i)$  by
\begin{equation}\label{eq:FS.FA02}
\Pr_i f(x)=\begin{cases}
\frac{1}{n_i}\sum\limits_{w=1}^{n_i}f(\phi_i(x)w)&\text{if }x\in F_i\setminus B_i,\\
f(x)&\text{if }x\in B_i.
\end{cases}
\end{equation}
The orthogonal complement operator, $\Pr_i^\bot\colon C(F_i)\to L^2_{\asym}(F_i,\mu_i)\cap C(F_i)$, is defined as 
\begin{equation*}\label{eq:FS.FA03}
\Pr_i^\bot f(x)=f(x)-\Pr_i f(x).
\end{equation*}
Analogous formal definitions of these operators applies to bounded Borel functions.
\end{definition}

The cell structure of $F_i$ described in Definition~\ref{def:DF.CS01} provides a natural framework to study a canonical diffusion on the space by means of cable systems/quantum graphs. We refer to~\cite{BB04,Kuc04} for basic definitions and results concerning them. To this end, for any $i\geq 1$ and $\alpha=\alpha_1\ldots \alpha_{i+1}\in\mcA_i$, we set $f_\alpha:=\varphi_\alpha^*f$. Taking into account the convention imposed in Remark~\ref{R:DF.CS01}, any function $f\in L^2(F_i,\mu_i)\cap C(F_i)$ can be viewed as a function $f\in\bigoplus_{\alpha\in\mcA_i}C([0,L_i])$ that satisfies the matching conditions
\begin{equation}\label{E:matching}
\begin{cases}
f_{\alpha}(0)=f_{\beta}(0),f_{\alpha}(L_i)=f_{\beta}(L_i) &\text{if }\alpha_\ell=\beta_\ell\quad\forall\,1\leq\ell\leq i,\\
f_\alpha(0)=f_\beta(L_i)&\text{if }\alpha_1\equiv (\beta_1\!+\!L_i)\hspace*{-.75em}\mod 2\pi,~\alpha_\ell=\beta_\ell\quad\forall\,1<\ell\leq i.
\end{cases}
\end{equation}
Figure~\ref{F:matching} illustrates these conditions.
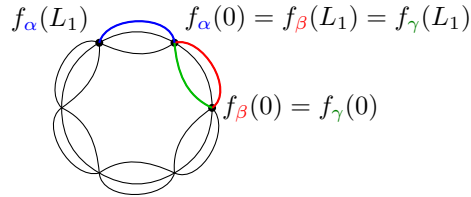
\begin{figure}[H]
\centering
\begin{tikzpicture}
\coordinate[label=above left:{\small{$f_{{\color{blue}\alpha}}(L_1)$}}] (x) at ($(120:1)$);
\fill (x) circle (1.5pt);
\coordinate[label=above right:{\small{$f_{{\color{blue}\alpha}}(0)=f_{\color{red}\beta}(L_1)=f_{\color{green!50!black}\gamma}(L_1)$}}] (z) at ($(60:1)$) ;
\fill (z) circle (1.5pt);
\coordinate[label=right:{\small{$f_{\color{red}\beta}(0)=f_{\color{green!50!black}\gamma}(0)$}}] (y) at ($(0:1)$);
\fill (y) circle (1.5pt);
\foreach \a in {0,60,120,180,240,300}{
\draw ($(\a:1)$) to[out=90-55+\a,in=15+\a] ($(60+\a:1)$);
\draw ($(\a:1)$) to[out=90+\a,in=330+\a] ($(60+\a:1)$);
\draw ($(\a:1)$) to[out=90+65+\a,in=-90+5+\a] ($(60+\a:1)$);
}
\draw[red, thick] ($(0:1)$) to[out=90-55,in=15] ($(60:1)$); 
\draw[green!75!black, thick] ($(0:1)$) to[out=90+65,in=-90+5] ($(60:1)$);
\draw[blue, thick] ($(60:1)$) to[out=90-55+60,in=15+60] ($(120:1)$);
\end{tikzpicture}
\caption{Matching conditions: $\beta_1=\gamma_1$, $\beta_2\neq\gamma_2$, and $\alpha_1=\beta_1+L_1$, $\alpha_2=\beta_2$.}
\label{F:matching}
\end{figure}
Using this notation, the symmetric and antisymmetric part of $L^2(F_i,\mu_i)$ can be characterized as follows.
\begin{proposition}\label{P:FS.FA01}
For any $i\geq 1$, 
\begin{align*}\label{eq:FS.FA06}
&(i)\; L^2_{\sym}(F_i,\mu_i)\cap C(F_i)=\Big\{f\in\!\bigoplus_{\alpha\in\mcA_i} C(F_i)~|~ f_{\alpha}(\theta)=f_{\beta}(\theta),~\theta\in [0,L_i],~\alpha_\ell=\beta_\ell\;\;\forall\,1\leq\ell\leq i\Big\}.\\
&(ii)\; L^2_{\asym}(F_i,\mu_i)\cap C(F_i)=\Big\{f\in\!\bigoplus_{\alpha\in\mcA_i} C(F_i)~|~\text{for any }x\in F_i, \sum_{\alpha\in\mcA_x}f_\alpha(\theta)=0~~\forall\,\theta\in [0,L_i]\Big\}.
\end{align*}
\end{proposition}
\begin{proof}
By construction, if two $i$-cells $I_{\alpha}$ and $I_{\beta}$ share both endpoints, then $\alpha_\ell=\beta_\ell$ for all $1\leq \ell\leq i$ and $\alpha_{i+1}\neq\beta_{i+1}$. Thus, for each $\theta\in[0,L_i]$, $\varphi_{\alpha}(\theta)=\phi_i(x)\alpha_{i+1}\in F_i\setminus B_i$ and $\varphi_\beta(\theta)=\phi_i(x)\beta_{i+1}\in F_i\setminus B_i$, for some $x\in F_{i-1}$, thus proving (i). On the other hand, $f\in L^2_{\asym}(F_i,\mu_i)\cap C(F_i)$ if and only if $\Pr_i^\bot f=f$, hence for each fixed $x\in F_i$ and any $\theta\in [0,L_i]$
\begin{equation*}\label{eq:FS.FA08}
\sum_{\alpha\in\mcA_x}f_\alpha(\theta)=\sum_{\alpha\in\mcA_x}\varphi_\alpha^*f(\theta)=\sum_{\alpha\in\mcA_x}f(\phi_i(x)\alpha_{i+1})=\sum_{w=1}^{n_i}f(\phi_i(x)w)=0
\end{equation*}
which proves~(ii).
\end{proof}
The preceding characterization has several useful implications. For instance, any function $\Pr_i f$ with $f\in C(F_i)$ can be regarded as a function in $C(F_{i-1})$.
\begin{corollary}\label{cor:FS.FA01}
\begin{enumerate}[leftmargin=*,label=(\roman*), align=left, labelsep=0in,wide=0em]
\item The subspace $L^2_{\sym}(F_i,\mu_i)\cap C(F_i)$ is isomorphic to $C(F_{i-1})$. Moreover, for each $x\in F_i$,
\begin{equation}\label{E:FS.FA09}
\sum_{\alpha\in\mcA_x}\varphi_\alpha^*f(\theta)=n_i\varphi_{\alpha_x}^*f(\theta)
\end{equation}
for any $\theta\in[0,L_i]$ and $f\in L^2_{\sym}(F_i,\mu_i)\cap C(F_i)$.
\item For any $f\in L^2_{\asym}(F_i,\mu_i)\cap C(F_i)$, it holds that $f(x)=0$ for all $x\in B_{i}$.
\end{enumerate}
\end{corollary}
Analogous statements hold with $\mcB_b(F_i)$ instead of $C(F_i)$.
\subsection{Diamond fractals} As an inverse limit space, a diamond fractal $F_\infty$ is naturally equipped with the Borel regular measure $\mu_\infty$ induced by the family of measures $\{\mu_i\}_{i\geq 0}$ from Definition~\ref{def:DF.CS03}. As a direct consequence of Definition~\ref{def:Pre02}, $L^2(F_\infty,\mu_\infty)$ admits the following dense subspace.
\begin{proposition}\label{P:FS.D01}
The subspace $\bigcup_{i\geq 0}\Phi_i^*C(F_i)$ is dense in $C(F_\infty)$ with respect to the uniform norm.
\end{proposition}
\begin{proof}
By construction, $C(F_\infty):=\{f\colon F_\infty\!\!\to\mbbR~|~\exists~\{f_i\}_{i\geq 0},\text{ with } f_i\in C(F_i)\text{ and }f(x)=\lim\limits_{i\to_\infty}f_i(\Phi_i(x))~\forall x\in F_\infty\}$.
\end{proof}
\section{Heat semigroup on finite approximations}\label{sec:4}
In this section we prove Theorem~\ref{T:MR02} by establishing the existence of a natural diffusion process on $F_i$ whose associated semigroup has the kernel $p_t^{F_i}(x,y)$. To this end we follow the lines of~\cite{BPY89} and prove in Proposition~\ref{P:HSFA01} that
\begin{equation}\label{eq:HSFA00}
T_t^{F_i}f(x)=\int_{F_i}p_t^{F_i}(x,y)f(y)\,\mu_i(dy)
\end{equation}
defines a strongly continuous Markov semigroup on $C(F_i)$ with the strong Feller property. From the general theory of Markov processes, see e.g.~\cite[Theorem A.2.2]{FOT94}, there exists a Hunt process on $F_i$ with transition function $p_t^{F_i}(x,y)$. The proof builds on an inductive argument that comes down to proving the case $i=1$.

To begin with, let $\{T^{F_0}_t\}_{t\geq 0}$ denote the Markov semigroup associated with the standard Brownian motion on the circle parametrized by $\varphi_{\text{\o}}(\theta)= e^{\text{i}\theta}$, so that $T_0=\id$ and, for each $t>0$,
\begin{equation*}\label{eq:HSFA01}
T^{F_0}_t g(\theta_0)=\int_0^{2\pi}p_t^{F_0}(\theta_0,\theta)g(\theta)\,d\theta
\end{equation*}
for any $g\in L^2([0,2\pi],\mu_{\text{\o}})\cap C([0,2\pi])$, where $p_t^{F_0}(x,y)$ is given by~\eqref{E:HKS1}.

For each $i\geq 1$, denote by $\{T_t^{[0,L_i]_D}\}_{t\geq 0}$ the Markov semigroup associated with one-dimensional Brownian motion on the interval $[0,L_i]$ killed at the boundary, i.e. $T^{[0,L_i]_D}_0=\id$ and, for each $t>0$,
\begin{equation*}\label{eq:HSFA02}
T_t^{[0,L_i]_D}g(\theta_0)=\int_0^{L_i}p_t^{[0,L_i]_D}(\theta_0,\theta)g(\theta)\,d\theta
\end{equation*}
for any $g\in L^2([0,L_i],d\theta)\cap C([0,L_i])$ with $g(0)=g(L_i)=0$, where $p_t^{[0,L_i]_D}(\theta_0,\theta)$ is given by~\eqref{E:1dHK}.

The key idea to prove Proposition~\ref{P:HSFA01} and hence obtain the formula in Theorem~\ref{T:MR02} relies in the decomposition of $T_t^{F_i}$ into two parts 
whose properties will be easier to analyze. 
These parts can be understood as the symmetric and antisymmetric, respectively. This decomposition is presented in the next lemma and it is crucial for the forthcoming proofs.
\begin{lemma}\label{L:HSFA01}
For any $i\geq 1$, $t\geq 0$, $f\in\mcB_b(F_i)$ and any fixed $x\in F_i$ it holds that
\begin{equation*}\label{eq:HSFA04}
T^{F_i}_tf(x)=T_t^{F_{i-1}}(\Pr_i f)(\phi_i(x))+T_t^{[0,L_i]_D}(\Pr_i^\bot\! f)_{\alpha_x}(\theta_x).
\end{equation*}
\end{lemma}
\begin{proof}
Let $f\in\mcB_b(F_i)$ and $x\in F_i$. Writing $f=\Pr_i f+\Pr_i^\bot\!f\in L_{\sym}^2(F_i,\mu_i)\oplus  L_{\asym}^2(F_i,\mu_i)$ one can split $T^{F_i}_tf(x)$ into two integrals. For the symmetric part, we have
\begin{equation}\label{eq:HSFA08}
\int_{F_i}p_t^{F_i}(x,y)\Pr_i f(y)\mu_i(dy)=\sum_{\alpha\in\mcA_i}\int_{I_\alpha}p_t^{F_i}(x,y)\Pr_i f(y)\mu_i(dy).
\end{equation}
Notice that in the formula from Theorem~\ref{T:MR02}, $p_t^{F_i}(x,y)$ has a different expression depending on the type of cell $y$ belongs to. In particular, whether that cell is the one containing $x$, $I_{\alpha_x}$, or it belongs to the same bundle as $I_{\alpha_x}$, or to a different bundle. Taking this on account and using~\eqref{E:MR02b}, the integral~\eqref{eq:HSFA08} becomes 
\begin{align}\label{eq:HSFA09}
\int_{F_i}p_t^{F_{i-1}}(\phi_i(x),\phi_i(y))\Pr_i f(y)\,\mu_i(dy)&-\frac{1}{n_i}\int_0^{L_i}\!p_t^{[0,L_i]_D}(\theta_x,\theta)\!\!\!\sum_{\alpha\in\mcA_x}(\Pr_i f)_{\alpha}(\theta)\,d\theta\\
&+\int_0^{L_i}\!p_t^{[0,L_i]_D}(\theta_x,\theta)(\Pr_i f)_{\alpha_x}(\theta)\,d\theta.\nonumber
\end{align}
Since $\Pr_i f\in L^2_{\sym}(F_i,\mu_i)$, the last two terms cancel out by Corollary~\ref{cor:FS.FA01}(i) and Lemma~\ref{lem:DF.CS01} yields
\begin{equation*}
\int_{F_i}p_t^{F_i}(x,y)\Pr_i f(y)\mu_i(dy)=\int_{F_{i-1}}p_t^{F_{i-1}}(\phi_i(x),y)\Pr_i f(y)\,\mu_{i-1}(dy)=T_t^{F_{i-1}}(\Pr_i f)(\phi_i(x))
\end{equation*}
as we wanted to prove. For the antisymmetric part analogous arguments yield~\eqref{eq:HSFA09} with $\Pr_i^\bot\! f$ instead of $\Pr_i f$. Now, since $p_t^{F_{i-1}}(\phi_i(x),\cdot)\in L^2_{\sym}(F_i,\mu)$, the first term in the corresponding version of~\eqref{eq:HSFA09} is zero. Furthermore, in view of Proposition~\ref{P:FS.FA01}(ii), the second term is zero as well and therefore
\begin{equation*}
\int_{F_i}p_t^{F_i}(x,y)\Pr_i^\bot f(y)\mu_i(dy)=\int_0^{L_i}\!p_t^{[0,L_i]_D}(\theta_x,\theta)(\Pr_i f)_{\alpha_x}(\theta)\,d\theta=T_t^{[0,L_i]_D}(\Pr_i^\bot\! f)_{\alpha_x}(\theta_x)
\end{equation*}
as desired.
\end{proof}

\begin{remark}\label{R:FS.FA02}
Referring to each part as symmetric, respectively antisymmetric, is consistent with the fact that, as functions on $F_i$, $T_t^{F_{i-1}}(\Pr_i f)(\phi_i(\cdot))\in L^2_{\sym}(F_i,\mu_i)\cap\mcB_b(F_i)$ and $T_t^{[0,L_i]_D}(\Pr_i^\bot\! f)_{\alpha_\cdot}(\theta_\cdot)\in L^2_{\asym}(F_i,\mu_i)\cap\mcB_b(F_i)$. In particular the latter holds because by virtue of Proposition~\ref{P:FS.FA01}
\begin{align*}
\sum_{\alpha\in\mcA_x}T_t^{[0,L_i]_D}(\Pr_i^\bot\! f)_{\alpha}(\theta)=\int_0^{L_i}\!p_t^{[0,L_i]_D}(\theta_x,\theta)\!\!\!\sum_{\alpha\in\mcA_x}(\Pr_i f)_{\alpha}(\theta)\,d\theta=0
\end{align*}
for any $x\in F_i$.
\end{remark}
The advantage of the decomposition from Lemma~\ref{L:HSFA01} unfolds in the case $i=1$ because the semigroups $T_t^{F_0}$ and $T_t^{[0,L_1]_D}$ are well-studied in the literature, see e.g.~\cite[Section 2.5]{BGL14}. This will be specially useful to establish by induction the validity of the next proposition.

\begin{proposition}\label{P:HSFA01}
The family of operators $\{T_t^{F_i}\}_{t\geq 0}$ defined by~\eqref{eq:HSFA00} is a strongly continuous conservative Markov semigroup on $C(F_i)$ with the strong Feller property.
\end{proposition}
\begin{proof}
We check the corresponding properties of the family $\{T_t^{F_i}\}_{t\geq 0}$ for $i=1$. 
\begin{enumerate}[leftmargin=.25in, label=(\roman*),wide=0in]
\item Symmetry: follows from the definition.
\item Semigroup property: On the one hand, applying Lemma~\ref{L:HSFA01} and the semigroup property of $T_t^{F_0}$ and $T_t^{[0,L_1]_D}$ it follows that for any $t,s> 0$
\begin{equation}\label{eq:HSFA11}
T_{t+s}^{F_1}f(x)=T_{t}^{F_0}T_s^{F_0}\Pr_1 f(\phi_1(x))+T_t^{[0,L_1]_D}P_s^{[0,L_1]_D}(\Pr_1^\bot\! f)_{\alpha_x}(\theta_x).
\end{equation}
On the other hand, applying twice the decomposition from Lemma~\ref{L:HSFA01} yields
\begin{align}\label{eq:HSFA12}
T_t^{F_1}T_s^{F_1}f(x)&=T_t^{F_0}\Pr_1\big(T_s^{F_0}\Pr_1 f(\phi_1(\cdot))\big)(x)+T_t^{F_0}\Pr_1\big(T_s^{[0,L_1]_D}(\Pr_1^\bot\!f)_{\alpha_\cdot}(\theta_\cdot)\big)(x)\\
&+T_t^{[0,L_1]_D}\big(\Pr_1^\bot T_s^{F_0}\Pr_1 f(\phi_1(\cdot))\big)_{\alpha_x}(\theta_x)+T_t^{[0,L_1]_D}T_s^{[0,L_1]_D}(\Pr_1^\bot\! f)_{\alpha_x}(\theta_x).\nonumber
\end{align}
In view of Remark~\ref{R:FS.FA02}, $\Pr_1\big(T_s^{[0,L_1]_D}(\Pr_1^\bot\!f)_{\alpha_\cdot}(\theta_\cdot)\big)$ and $\Pr_1^\bot\!\big(T_s^{F_0}\Pr_1 f(\phi_1(\cdot))\big)$ equal zero, hence~\eqref{eq:HSFA11} and~\eqref{eq:HSFA12} coincide.
\item Contractivity: Consider $f\in C(F_1)$ such that $f\geq 0$. By definition of the projector operator~\eqref{eq:FS.FA02}, we have that $\Pr_1 f\geq 0$ and the properties of $T_t^{F_0}$ thus yield $T_t^{F_0}\Pr_1 f\geq 0$. Furthermore, $\Pr_1$ is a contraction (see e.g.~\cite[Theorem 3, p.84]{Yos74}), which implies $\Pr_1^\bot f\geq 0$ as well. By the properties of $T_t^{[0,L_1]_D}$ we have $T_t^{[0,L_1]_D}(\Pr_1^\bot\!f)_{\alpha}\geq 0$ for any $\alpha\in\mcA_1$ and by linearity, $T_t^{F_1}f\geq 0$. The contractivity of $T_t^{F_0}$ and $T_t^{[0,L_1]_D}$ yields $\|T_t^{F_1}f\|_\infty\leq \|f\|_\infty$. This extends to any $f\in C(F_1)$ by decomposition into positive and negative part. 
\item Mass conservation: Since $\bm{1}\in L_{\sym}^2(F_1,\mu_1)$, $T_t^{F_1}\bm{1}=T_t^{F_0}\bm{1}=\bm{1}$.
\item Markov property: Follows from (iii) and (iv).
\item Strong continuity: For any $x\in F_1$ we have that $(\Pr_1^\bot\!f)_{\alpha_x}(\theta_x)=\Pr_1^\bot\!f(x)$ and $\Pr_1 f(x)$ can be identified with $\Pr_1 f(\phi_i(x))$. Thus, 
\begin{equation*}
|T_t^{F_1}f(x)-f(x)|\leq |T_t^{F_0}\Pr_1 f(\phi_1(x))-\Pr_1 f(x)|+|T_t^{[0,L_1]_D}(\Pr_1^\bot\!f)_{\alpha_x}(\theta_x)-\Pr_1^\bot\!f(x)|,
\end{equation*}
which tends to zero as $t\to 0$ by virtue of the strong continuity of $T_t^{F_0}$ and $T_t^{[0,L_1]_D}$.
\item Strong Feller property: Applying Lemma~\ref{L:HSFA01} and the triangular inequality, for any $f\in\mcB_b(F_1)$,
\begin{multline}\label{eq:HSFA10}
|T_t^{F_1}f(x)-T_t^{F_1}f(y)|
\leq |T_t^{F_0}\Pr_i f(\phi_1(x))-T_t^{F_0}\Pr_1 f(\phi_1(y))|\\
+|T_t^{[0,L_1]_D}(\Pr_1^\bot\!f)_{\alpha_x}(\theta_x)-T_t^{[0,L_1]_D}(\Pr_1^\bot\!f)_{\alpha_y}(\theta_y)|.
\end{multline}
On the one hand, $(P_1^\bot f)_\alpha\in\mcB_b([0,L_1])$ for any $\alpha\in\mcA_1$ and by Corollary~\ref{cor:FS.FA01}, $\Pr_1 f$ can be viewed as a function in $\mcB_b(F_0)$. On the other hand, if $x$ is close enough to $y$, then $\phi_1(x)$ is close to $\phi_1(y)$ and $\theta_x$ is close to $\theta_y$. Moreover, there exists $\alpha\in \mcA_1$ such that $x,y\in I_\alpha$, so that $\alpha_x=\alpha_y=\alpha$. The strong Feller property of $T_t^{F_0}$ and of $T_t^{[0,L_1]_D}$ now imply that~\eqref{eq:HSFA10} vanishes as $x$ approaches $y$.
\end{enumerate}
Finally, let us consider $i>1$ and assume that $\{T_t^{F_{i-1}}\}_{t\geq 0}$ is a strongly continuous conservative Markov semigroup on $C(F_{i-1})$ with the strong Feller property. The previous arguments applied verbatim substituting $F_0$ by $F_{i-1}$ and $L_1$ by $L_i$ prove the assertion for $\{T_t^{F_i}\}_{t\geq 0}$.
\end{proof}

As a consequence of the preceding result, the function $p^{F_i}_t(x,y)$ defined in Theorem~\ref{T:MR02} is a symmetric Markovian transition function. By density of continuous functions in $L^2(F_i,\mu_i)$ one obtains a family of operators $\{P_t^{F_i}\}_{t\geq 0}$ in $L^2(F_i,\mu_i)$ with the same properties as $\{T^{F_i}_t\}_{t\geq 0}$. Due to the strong continuity, there exists a unique Dirichlet form $(\mcE^{F_i},\mcF^{F_i})$ on $L^2(F_i,\mu_i)$ and a diffusion process $\{X^{F_i}_t\}_{t> 0}$ on $F_i$ whose associated heat kernel is $p_t^{F_i}(x,y)$, see e.g.~\cite[Lemma 1.3.1, Lemma 1.3.2, Theorem 1.3.1]{FOT94}.

\begin{remark}\label{R:HSFA02}
\begin{enumerate}[leftmargin=.25in, label=(\roman*),wide=0in]
\item Regarding $F_i$ as a cable system/quantum graph~\cite{BB04,Kuc04}, the Dirichlet form $(\mcE^{F_i},\mcF^{F_i})$ is given by
\begin{equation}\label{E:HSFA14}
\mcE^{F_i}(f,g)=\sum_{\alpha\in\mcA_i}\frac{1}{N_i}\int_0^{L_i}f_\alpha'(\theta)g_\alpha'(\theta)\;d\theta,\qquad f,g\in\mcF^{F_i},
\end{equation}
where the derivatives are understood in the weak sense, and 
\begin{equation*}
\mcF^{F_i}=\Big\{f\in\bigoplus_{\alpha\in\mcA_i}H^1([0,L_i],d\theta)~|~f\text{ satisfies the matching contidions~\eqref{E:matching}}\Big\}.
\end{equation*}
\item Since $F_i$ is compact, closed balls are also compact and the Dirichlet form $(\mcE^{F_i},\mcF^{F_i})$ regular. By~\cite[Theorem 10.4]{Kig12}, $p_t^{F_i}(x,y)$ is jointly continuous and hence uniformly continuous on $F_i\times F_i$. This fact can also be proved directly using the preceding results.
\end{enumerate}
\end{remark}

\section{Heat semigroup on diamond fractals}\label{sec:5}
Once the diffusion processes $\{X^{F_i}_t\}_{t\geq 0}$, $i\geq 1$, have been identified, the existence of a Markov process on $F_\infty$ follows from~\cite[Theorem 4.3]{BE04}. In this section our aim is to identify the heat kernel associated with that limiting process. Firstly, we establish the validity of Theorem~\ref{thm:MR03} by proving that the associated semigroup corresponds to a diffusion on $F_\infty$. Secondly, for the sake of completeness, we briefly 
discuss the Dirichlet form and properties of the heat kernel in the regular $2$-$2$ diamond fractal (i.e.\ $n_i=2=j_i$) which are known in the literature. 
Recall that in this section, the parameter sequences $\mcJ,\mcN$ under consideration satisfy the condition~\eqref{E:Ass_NJ}.

\begin{remark}\label{R:HSD01}
Assumption~\eqref{E:Ass_NJ} is equivalent to $\lim_{i\to\infty}N_iJ_ie^{-J_i^2t}=0$ for all $t>0$. Thus, Corollary~\ref{C:MR01} implies that for each fixed $t>0$, $p_t^{F_\infty}(x,y)$ is the uniform limit of uniformly continuous functions and therefore it is uniformly continuous on $F_\infty\times F_\infty$. In addition,~\eqref{E:Ass_NJ} implies that
\begin{equation*}
\sum_{i=1}^\infty N_iJ_i\big(1+(J_i^2t)^{-1}\big)e^{-J_i^2t}<\infty
\end{equation*}
for all $t>0$, hence $p_t^{F_\infty}(x,y)$ is jointly continuous on $(t_1,t_2)\times F_\infty\times F_\infty$ for any $0<t_1<t_2<\infty$. This can be seen by writing $p_t^{F_\infty}(x,y)$ as a telescopic series. 
\end{remark}
\subsection{Semigroup} 
For each $t>0$, any $f\in C(F_\infty)$ and $x\in F_\infty$ define
\begin{equation}\label{eq:HSD01}
T^{F_\infty}_tf(x):=\int_{F_\infty}p_t^{F_\infty}(x,y)f(y)\,\mu_{\infty}(dy),
\end{equation}
where $p_t^{F_\infty}(x,y)$ is given by~\eqref{eq:MR03}. Further, set $T^{F_\infty}_0=\id$. The proof that these operators define a suitable semigroup relies on the following lemma, that relates semigroups from different approximation levels, and ultimately the operator $T^{F_\infty}_t$ resembling~\cite[equation (4.2)]{BE04}.

\begin{lemma}\label{L:HSD01}
Let $\mcJ$ and $\mcN$ satisfy~\eqref{E:Ass_NJ}. For any $t>0$ and $i\geq 0$ it holds that
\begin{equation}\label{E:HSD02}
T_t^{F_i}\phi^*_{ik}=\phi^*_{ik}T_t^{F_k}
\end{equation}
for any $0\leq k<i$, and
\begin{equation}\label{E:HSD03}
T_{t}^{F_\infty}\Phi_{i}^*=\Phi_i^*T_t^{F_i}.
\end{equation}
\end{lemma}
\begin{proof}
We start by proving~\eqref{E:HSD02}. Let $f\in C(F_k)$ and notice that $\phi^*_{ik}f=\phi^*_i(\phi^*_{(i-1)k}f)$ belongs to $L^2_{\sym}(F_i)\cap C(F_i)$. By Corollary~\ref{cor:FS.FA01}, the function $\Pr_i \phi^*_{ik}f$ can thus be identified with $\phi^*_{(i-1)k}f\in C(F_{i-1})$. Applying Lemma~\ref{L:HSFA01}, for any $x\in F_i$ we have that
\begin{equation*}\label{eq:HSFA15}
T_t^{F_i}\phi^*_{ik}f(x)=T_t^{F_{i-1}}\big(\Pr_i\phi^*_{ik}f \big)(\phi_i(x))=T_t^{F_{i-1}}\big(\phi^*_{(i-1)k}f\big)(\phi_i(x)).
\end{equation*}
Iterating this argument with decreasing indices yields 
\begin{equation*}\label{eq:HSFA16}
T_t^{F_i}\phi^*_{ik}f(x)=T_t^{F_k}f(\phi_{k+1}{\circ}\cdots{\circ}\phi_i(x))=\phi^*_{ik}T_t^{F_k}f(x).
\end{equation*}
In order to show~\eqref{E:HSD03}, let us consider first $h\in C(F_i)$ and $x\in F_\infty$. Due to the uniform continuity of $p_t^{F_\infty}(x,y)$, Remark~\ref{rem:Pre01} and the construction of $\mu_\infty$ we have that
\begin{align*}
T_t^{F_\infty}(\Phi_i^*h)(x)&=\int_{F_\infty}p_t^{F_\infty}(x,y)h(\Phi_i(y))\,\mu_\infty(dy)\\
&=\lim_{k\to\infty}\int_{F_\infty}p_t^{F_k}(\Phi_k(x),\Phi_k(y))h(\phi_{ki}{\circ}\Phi_k(y))\,\mu_\infty(dy)\\
&=\lim_{k\to\infty}\int_{F_k}p_t^{F_k}(\Phi_k(x),z)h(\phi_{ki}(z))\,\mu_k(dz)=\lim_{k\to\infty}T_t^{F_k}(\phi_{ki}^*h)(\Phi_k(x)).
\end{align*}
Finally,~\eqref{E:HSD02} yields $T_t^{F_k}(\phi_{ki}^*h)(\Phi_k(x))=\phi_{ki}^*(T_t^{F_i}h)(\Phi_k(x))$ and therefore
\begin{align*}
T_t^{F_\infty}(\Phi_i^*h)(x)&=\lim_{k\to\infty}\phi_{ki}^*(T_t^{F_i}h)(\Phi_k(x))=\lim_{k\to\infty}T_t^{F_i}h(\phi_{ki}{\circ}\Phi_k(x))\\
&=\lim_{k\to\infty}T_t^{F_i}h(\Phi_i(x))=T_t^{F_i}h(\Phi_i(x))=\Phi^*_i(T_t^{F_i}h)(x).
\end{align*}
By virtue of Proposition~\ref{P:FS.D01}, the same holds by density for any $h\in C(F_\infty)$.
\end{proof}

\begin{proposition}\label{thm:HSD01}
The family of operators $\{T^{F_\infty}_t\}_{t\geq 0}$ is a strongly continuous Markov semigroup on $C(F_\infty)$ that satisfies the strong Feller property.
\end{proposition}
\begin{proof}
Let us first prove the semigroup property. From Proposition~\ref{L:HSD01} and Theorem~\ref{P:HSFA01} we have that
\begin{equation*}\label{eq:HSD06}
T_{t+s}^{F_\infty}\Phi^*_i=\Phi^*_iT_{t+s}^{F_i}=\Phi^*_i(T_t^{F_i}T_s^{F_i})=T_t^{F_\infty}(\Phi^*_iT_s^{F_i})=T_t^{F_\infty}T_s^{F_\infty}\Phi_i^*
\end{equation*}
holds for any $i\geq 0$ and $t,s> 0$. Hence, $T_{t+s}^{F_\infty}f=T_t^{F_\infty}T_s^{F_\infty}f$ holds for any $f\in \bigcup_{i\geq 0}\Phi_i^*C(F_i)$ and by density for any $f\in C(F_\infty)$. Applying the same arguments, i.e.\ density, Proposition~\ref{L:HSD01} and Theorem~\ref{P:HSFA01}, the strong continuity, contractivity and the Feller property of the semigroup $\{T_t^{F_\infty}\}_{t\geq 0}$ is deduced from the corresponding property of the approximations. 

Finally, the strong Feller property of the semigroup $\{T^{F_\infty}_t\}_{t\geq 0}$ is a consequence of the joint continuity of $p_t^{F_\infty}(x,y)$, c.f.\ Remark~\ref{R:HSFA02} (ii): for any $f\in\mcB_b(F_\infty)$ and $x,y\in F_\infty$,
\begin{align*}
|T^{F_\infty}_tf(x)-T^{F_\infty}_tf(y)|&\leq \int_{F_\infty}|p^{F_\infty}_t(x,z)-p^{F_\infty}_t(y,z)|\,f(z)\,\mu_\infty(dz)\\
&\leq \bigg(\int_{F_\infty}|p_t^{F_\infty}(x,z)-p_t^{F_\infty}(y,z)|^2\mu_\infty(dz)\bigg)^{1/2}\|f\|_{L^2(F_\infty,\mu_\infty)}\\
&=(p_{2t}^{F_\infty}(x,x)-p_{2t}^{F_\infty}(x,y)+p_{2t}^{F_\infty}(x,y)-p_{2t}^{F_\infty}(y,y))^{1/2}\|f\|_{L^2(F_\infty,\mu_\infty)}
\end{align*}
and therefore $T^{F_\infty}_tf\in C(F_\infty)$.
\end{proof}

By density of continuous functions in $L^2(F_\infty,\mu_\infty)$ one defines the family of operators $\{P_t^{F_\infty}\}_{t\geq 0}$ on $L^2(F_\infty,\mu_\infty)$ with the same properties as $\{T^{F_\infty}_t\}_{t\geq 0}$.
\begin{corollary}\label{C:HSD01}
The family of operators $\{P^{F_\infty}_t\}_{t\geq 0}$ is a strongly continuous Markov semigroup on $L^2(F_\infty,\mu_\infty)$ that satisfies the strong Feller property. 
\end{corollary}

\begin{remark}\label{R:HSD02}
The Dirichlet form associated with $\{P^{F_\infty}_t\}_{t\geq 0}$ is given by
\begin{align*}
&\mcE^{F_\infty}(f,f)=\lim_{t\to 0}\frac{1}{t}\langle f-P_t^{F_\infty}f,f\rangle_{L^2(F_\infty,\mu_\infty)}\\
&\mcF^{F_\infty}=\{f\in L^2(F_\infty,\mu_\infty)~|~\mcE^{F_\infty}(f,f) \text{ exists and is finite}\},
\end{align*}
see e.\ g.~\cite[Definition 1.7.1]{BGL14}. Using density, Proposition~\ref{P:FS.D01} and Lemma~\ref{L:HSD01}, this corresponds with the definition introduced in~\cite[Definition 4.1]{ST12} in the context of inverse limit spaces.
\end{remark}

\subsection{Standard 2-2 diamond fractal. Dirichlet form and heat kernel}
By~\cite[Theorem 1.4.3]{FOT94} the Markov process $X^{F_\infty}_t$ associated with the semigroup $\{P^{F_\infty}_t\}_{t\geq 0}$ is a Hunt process. The corresponding Dirichlet form $(\mcE^{F_\infty},\mcF^{F_\infty})$ was obtained in~\cite{HK10} in the standard $2$-$2$ diamond fractal following a different approach. Throughout this paragraph we thus set $j_i=n_i=2$ for all $i\geq 1$, so that $F_\infty$ coincides with the diamond fractal discussed in~\cite{HK10}. There, $F_\infty$ is approximated by a sequence of finite sets that correspond to $\{B_i\}_{i\geq 0}$ from Definition~\ref{def:DF.PL01}. We outline next the arguments leading to the fact that the Dirichlet form $(\mcE^{F_\infty},\mcF^{F_\infty})$ associated with the semigroup $\{P_t^{F_\infty}\}_{t\geq 0}$ coincides with the Dirichlet form $(\mcE,\mcF)$ constructed there.

\begin{definition}\label{def:DFHK01}
A function $h\in C(F_\infty)$ is called $i$-harmonic for some $i\geq 0$ if
\begin{itemize}[leftmargin=.3in]
\item[(i)] $h_{|_{F_i}}\in C(F_i)$ is the linear interpolation of $h_{|_{B_i}}\colon B_i\to\mbbR$,
\item[(ii)] $h=\Phi^*_ih_{|_{F_i}}$.
\end{itemize}
Further, define $\mcH_*=\{h\in C(F_\infty)~|~h\text{ is }i\text{-harmonic for some }i\geq 0\}$.
\end{definition}
By virtue of the maximum principle, $\mcH_*$ is a dense subspace of $C(F_\infty)$ and therefore a core of $(\mcE^{F_\infty},\mcF^{F_\infty})$: Due to the definition of the Dirichlet form given in Remark~\ref{R:HSD02}, it follows from Proposition~\ref{P:FS.D01} and Lemma~\ref{L:HSD01} that the completion of $\mcH_*$ with respect to $(\mcE^{F_\infty}(\cdot,\cdot)+\langle\cdot,\cdot\rangle_{L^2(F_\infty,\mu_\infty)})^{1/2}$ is $\mcF^{F_\infty}$.

On the one hand, in view of Remark~\ref{rem:Pre01}, for any $f\in\Phi_i^*C(F_i)$ there exists $f_i\in C(F_i)$ such that
\begin{equation}\label{eq:DFHK01}
f(x)=f_i(\Phi_i(x))=f_i{\circ}\phi_{ki}(\Phi_k(x))\in\Phi_k^*C(F_k)
\end{equation}
for any $k\geq i$. Thus, if $h\in\mcH_*$ is $i$-harmonic, Definition~\ref{def:DFHK01}, Remark~\ref{R:HSD02} and~\eqref{eq:DFHK01} yield
\begin{equation*}\label{eq:DFHK02}
\mcE^{F_\infty}(h,h)=\lim_{k\to\infty}\mcE^{F_k}(h_{|_{F_i}}{\circ}\phi_{ki},h_{|_{F_i}}{\circ}\phi_{ki})=\mcE^{F_i}(h_{|_{F_i}},h_{|_{F_i}})
\end{equation*}
with $\mcE^{F_i}$ as in~\eqref{E:HSFA14}. On the other hand, since $h_{|_{F_i}}$ is given by linear interpolation of $h_{|_{B_i}}$, also $h_{|_{F_i}}{\circ}\phi_{ki}$ is linear interpolation of $h_{|_{B_i}}{\circ}\phi_{ki}=h_{|_{B_k}}$ for any $k\geq i$. Thus,
\begin{equation*}\label{eq:DFHK03}
\mcE^{F_k}(h_{|_{B_i}}{\circ}\phi_{ki},h_{|_{B_i}}{\circ}\phi_{ki})=\mcE_k(h_{|_{B_k}},h_{|_{B_k}})
\end{equation*}
for any $k\geq i$, where $(\mcE_k,\ell(B_k))$ is the bilinear form defined in~\cite[Section 4]{HK10} and $\ell(B_k)=\{f\colon B_k\to\mbbR\}$. Hence, for any $h\in\mcH_*$
\begin{equation*}\label{eq:DFHK04}
\mcE^{F_\infty}(h,h)=\lim_{k\to\infty}\mcE^{F_k}(h_{|_{F_i}}{\circ}\phi_{ki},h_{|_{F_i}}{\circ}\phi_{ki})=\lim_{k\to\infty}\mcE_k(h_{|_{B_k}},h_{|_{B_k}})=\mcE(h,h).
\end{equation*}
Using Lemma~\ref{L:HSD01}, one can now show that both forms $\mcE^{F_\infty}$ and $\mcE$ coincide.

In~\cite[Section 4]{HK10}, a series of properties of the diffusion process $\{X^{F_\infty}_t\}_{t\geq 0}$ are shown, for instance Poincar\'e inequality, ultracontractivity and exit times, as well as the estimates
\begin{align*}
0<~p_t^{F_\infty}(x,y)\leq \frac{c_1}{t}\exp\bigg(-c_2\frac{d(x,y)^2}{t}\bigg),\qquad p_t^{F_\infty}(x,x)\geq \frac{c_3}{V(x,c_4\sqrt{t})}
\end{align*}
for all $x,y\in F_\infty$ and $t\in (0,1)$. On the negative side, elliptic Harnack inequality is not supported, see~\cite[Remark 4.15]{HK10}. 

We would like to emphasize that these results were obtained for the regular $2$-$2$ diamond and relied strongly on the self-similarity of the space. Even in this case, further inequalities of interest such as (weak) Bakry-\'Emery gradient estimates remain so far unknown. We hope that 
an explicit expression of the heat kernel can shed some light on these questions and lead to better estimates.
\section*{Acknowledgments} The author is greatly thankful to F.\ Baudoin, G.\ Dunne, M.\ Gordina and A.\ Teplyaev for very valuable discussions and comments.
\bibliographystyle{amsplain}
\bibliography{HKDF_Refs}
\end{document}